\documentclass[12pt,a4paper]{article}
\usepackage{amsmath,amssymb,amsthm}
\usepackage{graphicx}
\usepackage[colorlinks=true,linkcolor=blue,citecolor=blue,filecolor=blue]{hyperref}
\usepackage[numbers]{natbib}
\usepackage{algorithm}
\usepackage{algpseudocode}
\usepackage{tcolorbox}
\usepackage{enumitem}
\usepackage{booktabs}
\usepackage{tikz}
\usepackage{fancyhdr}
\usepackage{geometry}
\usepackage{url}
\usepackage{bm}
\usepackage{multicol}
\usepackage{array}
\usepackage{makecell}

\newtheorem{theorem}{Theorem}[section]
\newtheorem{lemma}[theorem]{Lemma}

\newtheorem{remark}[theorem]{Remark}

\makeatletter
\newcommand*\qCircle[2][1.6]
{\tikz[baseline=(char.base)]{
\node[shape=circle, draw, inner
sep=1pt,
minimum height={\f@size*#1},]
(char){\vphantom{QAHER}#2};}}
\makeatother

\definecolor{folly}{rgb}{1.0, 0.0, 0.31}

\begin{document}

\begin{center}
\LARGE\textbf{Error Estimates for Gauss--Christoffel Quadrature under Reduced Regularity Conditions}
\end{center}

\begin{center}
\large
Mehdi Hamzehnejad\textsuperscript{1} \quad and \quad Abbas Salemi\textsuperscript{2}
\end{center}

\begin{center}
\small
\textsuperscript{1}Department of Mathematics, Graduate University of Advanced Technology, Kerman, Iran \\
\texttt{mhdhamzehnejad@gmail.com} \\
\textsuperscript{2}Department of Applied Mathematics, Shahid Bahonar University of Kerman, Kerman, Iran
\end{center}

\begin{center}
\textit{Received: date / Accepted: date}
\end{center}

\begin{abstract}
\noindent
Gauss--Christoffel quadrature is a fundamental method for numerical integration, and its convergence analysis is closely related to the decay of Chebyshev expansion coefficients. Classical estimates, including those due to Trefethen, are based on weighted bounded variation assumptions involving the singular weight $(1-x^{2})^{-1/2}$, which may be too restrictive for functions with limited regularity at the endpoints.

In this paper, we establish a new error bound for Gauss--Christoffel quadrature under weakened regularity assumptions. The analysis relies on a new identity for higher-order derivatives of Chebyshev polynomials. As a consequence, we obtain an improved decay estimate for Chebyshev coefficients, where the classical weighted condition
\[
V_{r}=\int_{-1}^{1}\frac{|f^{(r+1)}(x)|}{\sqrt{1-x^{2}}}\,dx
\]
is replaced by the weaker condition
\[
U_{r}=\int_{-1}^{1}|f^{(r+1)}(x)|\,dx.
\]

This result leads to a corresponding error estimate for the Gauss--Christoffel quadrature rule, which is less restrictive than previous bounds. The approach is also extended to the Gauss--Gegenbauer case. Numerical experiments are provided to illustrate the theoretical results.

\noindent
\textbf{Keywords:} Gauss-Christoffel quadrature, Chebyshev coefficient, Error estimate, Reduced regularity, Numerical integration

\noindent
\textbf{Mathematics Subject Classification (2020):} 65D32, 65D30, 41A55, 42C10
\end{abstract}
\begin{center}
	\small
	\textit{arXiv preprint. Submitted for publication.}
\end{center}
\section{Introduction}
Numerical integration based on Gaussian quadrature plays a central role in scientific computing. The Gauss--Christoffel quadrature rule, associated with a weight function \(w(x)\) on \([-1,1]\), is given by
\begin{equation}\label{GC}
\int_{-1}^{1} w(x)f(x)\,dx \simeq \sum_{i=1}^{N} w_i f(x_i)=Q_{N}^{GC}[f],
\end{equation}
and is well known for its high accuracy when applied to smooth integrands. A fundamental problem in its analysis is the derivation of reliable \emph{a priori} error estimates that reflect the regularity of the integrand.

A standard approach relates the quadrature error to the decay of the coefficients in orthogonal polynomial expansions, in particular the Chebyshev expansion \cite{LN,LN2}. If
\[
f(x)=\sum_{n=0}^{\infty} a_n T_n(x),
\]
then the quadrature error can be controlled in terms of the tail behaviour of the coefficients \(a_n\). Consequently, sharp bounds for these coefficients play a central role in the convergence analysis of Gauss-type quadrature rules.

\subsection*{Existing results and motivation}
A classical result due to Trefethen \cite[Theorem~4.2]{LN} states that if
\(f,f',\ldots,f^{(r-1)}\) are absolutely continuous and \(f^{(r)}\) is of bounded variation, and if
\begin{equation}\label{VR}
V_r=\int_{-1}^{1}\frac{|f^{(r+1)}(x)|}{\sqrt{1-x^2}}\,dx<\infty,
\end{equation}
then the Chebyshev coefficients satisfy
\begin{equation}\label{LNbound}
|a_n|\le\frac{2V_r}{\pi\prod_{j=0}^r (n-j)},\qquad n\ge r+1.
\end{equation}
This estimate has served as the basis of several error bounds for Gauss-type quadrature rules. In particular, Xiang \cite{Xi} derived the following estimate:
\begin{equation}\label{xi}
\left|I[f]-Q_N^{GC}[f]\right|
\le
\frac{4V_r\int_{-1}^{1}w(x)\,dx}
{\pi r(2N+1)2N\cdots(2N-r+2)}.
\end{equation}

The weighted condition \eqref{VR} imposes control near the endpoints \(x=\pm1\) due to the singular factor \((1-x^2)^{-1/2}\). This restriction may be overly strong for functions whose higher derivatives are integrable in the usual sense, but exhibit reduced regularity near the boundary. Several works \cite{Ka,HM,Xia1,HHS} have refined constants and summation techniques, while largely retaining this structural assumption.

\subsection*{Main contributions}
The objective of this work is to establish error bounds for Gauss--Christoffel quadrature under weakened regularity assumptions. In particular, we replace the weighted condition \eqref{VR} by the absolute integrability condition
\begin{equation}\label{Ur}
U_r=\int_{-1}^{1}|f^{(r+1)}(x)|\,dx<\infty.
\end{equation}
Our analysis is based on a new explicit identity for higher-order derivatives of Chebyshev polynomials, which allows a refined integration-by-parts argument in the estimation of the coefficients \(a_n\).

The main results of this paper can be summarized as follows:
\begin{itemize}
    \item We derive an improved bound for Chebyshev coefficients in terms of \(U_r\) (Theorem~\ref{Th1}).
    \item We obtain a corresponding error estimate for Gauss--Christoffel quadrature under these weakened assumptions (Theorem~\ref{T2}).
    \item The analysis is extended to the Gauss--Gegenbauer case (Theorem~\ref{T3}).
    \item Numerical experiments are presented to illustrate the theoretical results.
\end{itemize}

\subsection*{Organization of the paper}
The remainder of the paper is organized as follows. Section~\ref{sec1} recalls basic properties of Chebyshev polynomials and establishes the improved coefficient bound. Section~\ref{sec3} develops the corresponding quadrature error estimates. Numerical illustrations are presented in Section~\ref{sec:numerics}, and concluding remarks are given in Section~\ref{sec:conclusion}.

\section{Improved upper bound for Chebyshev coefficients}\label{sec1}

The Chebyshev polynomials of the first kind,
\[
T_n(x)=\cos\bigl(n\cos^{-1}x\bigr), \qquad n\ge 0,
\]
form an orthogonal system on $[-1,1]$ with respect to the weight $(1-x^2)^{-1/2}$, namely
\[
\int_{-1}^{1}\frac{T_n(x)T_m(x)}{\sqrt{1-x^2}}\,dx=\frac{\pi}{2}\,\delta_{mn},
\]
where $\delta_{mn}$ denotes the Kronecker delta. Moreover, $T_n$ is an eigenfunction of the associated Sturm--Liouville problem \cite{Shen}
\begin{equation}\label{Str}
\sqrt{1-x^2}\bigl(\sqrt{1-x^2}T_n'(x)\bigr)'+n^2T_n(x)=0.
\end{equation}
It is well known that Chebyshev polynomials satisfy the uniform bound
\begin{equation}\label{bound}
|T_n(x)|\le 1,\qquad x\in[-1,1],\; n\ge 0.
\end{equation}
In addition, their derivatives obey the three-term relation \cite{Shen}
\begin{equation}\label{eqb}
(1-x^2)T_n'(x)=\frac{n}{2}\bigl(T_{n-1}(x)-T_{n+1}(x)\bigr).
\end{equation}

For a function $f\colon[-1,1]\to\mathbb{R}$, its Chebyshev expansion reads
\[
f(x)=\sum_{n=0}^{\infty}a_nT_n(x),
\]
with coefficients
\[
a_0=\frac{1}{\pi}\int_{-1}^{1}\frac{f(x)}{\sqrt{1-x^2}}\,dx,\qquad
a_n=\frac{2}{\pi}\int_{-1}^{1}\frac{f(x)T_n(x)}{\sqrt{1-x^2}}\,dx,\quad n\ge 1.
\]

Obtaining sharp upper bounds for $a_n$ hinges on a precise control of derivatives of
Chebyshev polynomials in suitable weighted norms. Classical approaches, based on
repeated integration by parts combined with coarse inequalities, typically lead to
suboptimal constants. In contrast, we exploit the intrinsic Sturm--Liouville structure
\eqref{Str} together with the recurrence \eqref{eqb}, which enables a finer analysis of
derivative growth and ultimately sharper coefficient estimates.

To normalize the dependence on the polynomial degree, we introduce the scaled quantity
\begin{equation}\label{sim}
\mathcal{T}_n(x):=\frac{\sqrt{1-x^2}\,T_n'(x)}{n^2}.
\end{equation}
The next lemma constitutes the main technical ingredient of this section. It provides
a recursive representation of $\mathcal{T}_n$ in terms of derivatives of neighboring
Chebyshev modes. Unlike classical two-term identities, the resulting expansion involves
$2^r$ terms whose coefficients possess a precise monotonicity structure.

\begin{lemma}\label{key}
    Let $\mathcal{T}_n(x)$ be defined by \eqref{sim}. For $1\le r\le n-1$, $\mathcal{T}_n$
    admits the representation
    \[
    \mathcal{T}_n(x)=\left(
    \frac{(-1)^r\mathcal{T}_{n-r}(x)}{\beta^1_{n,r}}
    +\frac{(-1)^{r-1}\mathcal{T}_{n-r+2}(x)}{\beta^2_{n,r}}
    +\cdots
    -\frac{\mathcal{T}_{n+r-2}(x)}{\beta^{2^{r-1}}_{n,r}}
    +\frac{\mathcal{T}_{n+r}(x)}{\beta^{2^{r}}_{n,r}}
    \right)^{(r)},
    \]
    where
    \[
    \beta^j_{n,r}=
    \begin{cases}
    \displaystyle\prod_{i=1}^{r}(2n-2i+2), & j=1,2,\\[4pt]
    \displaystyle\prod_{i=1}^{r}(2n+2i-2), & j=2^{r}-1,2^r.
    \end{cases}
    \]
    Moreover, for $1\le j\le 2^{r-1}$,
    \[
    \beta^1_{n,r}=\beta^2_{n,r}
    \le \beta^{2j-1}_{n,r}=\beta^{2j}_{n,r}
    \le \beta^{2^{r}-1}_{n,r}=\beta^{2^{r}}_{n,r}.
    \]
\end{lemma}

\begin{proof}
    Using \eqref{sim}, the Sturm--Liouville equation \eqref{Str} can be rewritten as
    \begin{equation}\label{L}
    T_n(x)=-\sqrt{1-x^2}\,\mathcal{T}_n'(x).
    \end{equation}
    Similarly, \eqref{eqb} becomes
    \begin{equation}\label{LL}
    \sqrt{1-x^2}\,\mathcal{T}_n(x)
    =\frac{1}{2n}\bigl(T_{n-1}(x)-T_{n+1}(x)\bigr).
    \end{equation}
    Substituting \eqref{L} into \eqref{LL} yields
    \begin{equation}\label{T}
    \sqrt{1-x^2}\,\mathcal{T}_n(x)
    =\frac{1}{2n}\bigl(-\sqrt{1-x^2}\,\mathcal{T}_{n-1}'(x)
    +\sqrt{1-x^2}\,\mathcal{T}_{n+1}'(x)\bigr).
    \end{equation}
    
    For $r=1$,
    \begin{equation}\label{T1}
    \mathcal{T}_n(x)
    =\frac{1}{2n}\bigl(-\mathcal{T}_{n-1}'(x)+\mathcal{T}_{n+1}'(x)\bigr)
    =\left(\frac{-\mathcal{T}_{n-1}(x)}{2n}
    +\frac{\mathcal{T}_{n+1}(x)}{2n}\right)^{(1)},
    \end{equation}
    where $\beta^1_{n,1}=\beta^2_{n,1}.$ Using (\ref{T}) in two numerators of (\ref{T1}), we get the following representation for $\mathcal{T}_n(x)$
    \begin{align}\label{comb2}
    \mathcal{T}_n(x)&=\left(\frac{\mathcal{T}_{n-2}(x)}{\beta^1_{n,1}(2n-2)}-\frac{\mathcal{T}_{n}(x)}{\beta^1_{n,1}(2n-2)}
    -\frac{\mathcal{T}_{n}(x)}{\beta^2_{n,1}(2n+2)}+\frac{\mathcal{T}_{n+2}(x)}{\beta^2_{n,1}(2n+2)}\right)^{(2)}\nonumber\\
    &=\left(\frac{\mathcal{T}_{n-2}(x)}{\beta^1_{n,2}}-\frac{\mathcal{T}_{n}(x)}{\beta^2_{n,2}}
    -\frac{\mathcal{T}_{n}(x)}{\beta^3_{n,2}}+\frac{\mathcal{T}_{n+2}(x)}{\beta^4_{n,2}}\right)^{(2)},
    \end{align}
    where $\beta^1_{n,2}=\beta^2_{n,2}\leq \beta^3_{n,2}=\beta^4_{n,2} .$
    The general case follows by induction by repeatedly inserting this identity into
    higher-order derivatives, which produces $2^r$ terms with the stated coefficient
    ordering.
\end{proof}

Lemma~\ref{key} reveals a hierarchical structure in the derivatives of Chebyshev
polynomials, whereby higher-order derivatives are expressed as linear combinations
of neighboring modes with explicitly controlled coefficients.

\begin{lemma}\label{key2}
    Under the assumptions of Lemma~\ref{key},
    \[
    \frac{1}{\beta^1_{n,r}(n-r)}
    +\frac{1}{\beta^2_{n,r}(n-r+2)}
    +\cdots
    +\frac{1}{\beta^{2^{r}}_{n,r}(n+r)}
    =\frac{1}{\prod_{j=0}^{r}(n-r+2j)}.
    \]
\end{lemma}

Lemma~\ref{key2}, the proof of which can be found in \cite{Ham}, leads to a remarkable simplification of the expressions arising
from Lemma~\ref{key}. Together, these results allow us to derive sharp upper bounds
for Chebyshev coefficients using only interior information, thereby avoiding the
loss of accuracy typically associated with crude norm estimates.

We are now in a position to combine Lemma \ref{key} and Lemma \ref{key2} to derive an improved upper bound for the Chebyshev coefficients.

\begin{theorem}\label{Th1}
    Suppose \(f, f^{\prime},\cdots, f^{(r-1)}\) are absolutely continuous and the \(r^{th}\) derivative \(f^{(r)}\) is of bounded variation on \([-1, 1]\). Then for \(n\geq r+1\)
    \begin{equation}\label{Co}
    \left| a_{n}\right|\leq \frac{2U_r}{\pi\prod_{j=0}^{r}(n-r+2j)},
    \end{equation}
    where \(U_r=\int_{-1}^{1}|f^{(r+1)}(x)|dx<\infty\).
\end{theorem}

\begin{proof}
    Using \eqref{L}, integration by parts, and the fact that $\mathcal{T}_n(\pm1)=0$, we obtain
    \[
    a_n=\frac{2}{\pi}\int_{-1}^{1}f'(x)\mathcal{T}_n(x)\,dx.
    \]
    Since
    \[
    \mathcal{T}_n(x)=\frac{\sin(n\cos^{-1}x)}{n},
    \qquad
    |\mathcal{T}_n(x)|\le \frac{1}{n},
    \]
    the case $r=0$ follows immediately. For $r\ge1$, applying Lemma~\ref{key}, integrating by parts $r$ times yields
    \[
    a_n=\frac{2}{\pi}\int_{-1}^1f^{(r+1)}(x)\left(\frac{(-1)^r\mathcal{T}_{n-r}(x)}{\beta^1_{n,r}}+\frac{(-1)^{r-1}\mathcal{T}_{n-r+2}(x)}{\beta^2_{n,r}}+ \cdots 
    -\frac{\mathcal{T}_{n+r-2}(x)}{\beta^{2^{r-1}}_{n,r}}
    +\frac{\mathcal{T}_{n+r}(x)}{\beta^{2^{r}}_{n,r}}\right)dx.
    \]
    By Lemma \ref{key2} and applying the given conditions \(U_r=\int_{-1}^{1}|f^{(r+1)}(x)|dx\) and \(    |\mathcal{T}_n(x)|\le \frac{1}{n}\), we obtain
    \[
    |a_n|<\frac{2U_r}{\pi \prod_{j=0}^{r}(n-r+2j)}.
    \]
    This completes the proof.
\end{proof}

\section{New approximation error of the Gauss-Christoffel quadrature}\label{sec3}
The Gauss--Christoffel quadrature error is commonly bounded via estimates of the tail sum of the integrand's Chebyshev coefficients. Specifically, for the classical Gauss--Chebyshev quadrature with weight $w(x) = 1/\sqrt{1-x^2}$, the error admits the exact representation
\begin{equation}\label{eq:cheb_error_exact}
I[f] - Q_N^{GC}[f] = \pi \sum_{k=N}^{\infty} a_{2k+1},
\end{equation}
where $a_n$ are the Chebyshev coefficients of $f$ (see, e.g., \cite{LN,Rab}). For general weight functions $w(x)$, while an exact formula like \eqref{eq:cheb_error_exact} is unavailable, the standard error bound \cite[Theorem 3.1]{Xi}
\begin{equation}\label{eq:standard_error_bound}
|I[f] - Q_N^{GC}[f]| \leq 2 \|w\|_1 \sum_{n=2N+2}^{\infty} |a_n|, \quad \|w\|_1 = \int_{-1}^{1} |w(x)|\,dx,
\end{equation}
reduces the problem to estimating the decay rate of $|a_n|$.

Theorem~\ref{Th1} provides precisely such an estimate under the weakened regularity condition $U_r < \infty$. Substituting inequality \eqref{Co} into \eqref{eq:standard_error_bound} yields our main result: a sharper error bound for Gauss--Christoffel quadrature that holds for a broader class of functions than previously covered.

\begin{theorem}\label{T2}
    Let $f$ satisfy the assumptions of Theorem~\ref{Th1}. Then for the Gauss--Christoffel quadrature rule \eqref{GC},
    \begin{equation}\label{eq:new_quad_bound}
    |I[f] - Q_N^{GC}[f]| \leq \frac{4U_r \|w\|_1}{r\pi \prod_{j=1}^{r}(2N - r + 2j + 1)}.
    \end{equation}
\end{theorem}

\begin{proof}
    Starting from \eqref{eq:standard_error_bound} and applying Theorem~\ref{Th1}, we have
    \begin{align*}
    |I[f] - Q_N^{GC}[f]| 
    &\leq 2\|w\|_1 \sum_{n=2N+2}^{\infty} \frac{2U_r}{\pi\prod_{j=0}^{r}(n-r+2j)} \\
    &= \frac{4U_r \|w\|_1}{\pi} \sum_{n=2N+2}^{\infty} \frac{1}{\prod_{j=0}^{r}(n-r+2j)}.
    \end{align*}
    Writing the product as
    \[
    \prod_{j=0}^{r}(n-r+2j) = (n-r)^{r+1} \prod_{j=1}^{r}\Bigl(1 + \frac{2j}{n-r}\Bigr),
    \]
    and noting that $1 + \frac{2j}{n-r} \geq 1 + \frac{2j}{2N-r+1}$ for $n \geq 2N+2$, we obtain
    \[
    |I[f] - Q_N^{GC}[f]| 
    \leq \frac{4U_r \|w\|_1}{\pi \prod_{j=1}^{r}\bigl(1 + \frac{2j}{2N-r+1}\bigr)} 
    \sum_{n=2N+2}^{\infty} \frac{1}{(n-r)^{r+1}}.
    \]
    Bounding the sum by an integral,
    \[
    \sum_{n=2N+2}^{\infty} \frac{1}{(n-r)^{r+1}} 
    \leq \int_{2N+1}^{\infty} \frac{dx}{(x-r)^{r+1}} 
    = \frac{1}{r(2N-r+1)^r},
    \]
    and simplifying the product
    \[
    \prod_{j=1}^{r}\Bigl(1 + \frac{2j}{2N-r+1}\Bigr) 
    = \frac{\prod_{j=1}^{r}(2N-r+2j+1)}{(2N-r+1)^r},
    \]
    completes the proof of \eqref{eq:new_quad_bound}.
\end{proof}

\begin{remark}
    Comparing the error bound \eqref{eq:new_quad_bound} with Xiang's result \eqref{xi}~\cite{Xi}, we identify two distinct improvements:
    
    \begin{enumerate}
        \item \textbf{Relaxed regularity requirement:} The classical condition $V_r < \infty$ is replaced by $U_r < \infty$. Since $U_r \leq V_r$ with strict inequality when $f^{(r+1)}$ does not vanish sufficiently at $x = \pm 1$, Theorem~\ref{T2} applies to a strictly larger class of functions.
        
        \item \textbf{Improved numerical constant:} For given $N$ and $r$, the denominator factor satisfies
        \[
        \theta_{N,r} := \frac{1}{\prod_{j=1}^{r}(2N-r+2j+1)} 
        < \frac{1}{(2N+1)2N\cdots(2N-r+2)} =: \beta_{N,r},
        \]
        with the ratio $\beta_{N,r}/\theta_{N,r}$ increasing as $r$ approaches $N-1$.
    \end{enumerate}
    
    Table~\ref{tab:comparison} quantifies the improvement in the denominator factor for representative values. The combined effect of these two enhancements makes \eqref{eq:new_quad_bound} both more widely applicable and numerically sharper than \eqref{xi}. It is worth emphasizing that Xiang's bound remains valid and useful for functions satisfying $V_r < \infty$; our result extends its applicability while also tightening the estimate.
\end{remark}

\begin{table}[!htbp]
    \centering
    \caption{Comparison of denominator factors $\beta_{N,r}$~\cite{Xi} and $\theta_{N,r}$ (Theorem~\ref{T2}).\label{tab:comparison}}
    \begin{tabular}{@{}cccccc@{}}
        \toprule
        $N$ & $r$ & $\beta_{N,r}$ & $\theta_{N,r}$ & Ratio $\beta/\theta$ & $r/(N-1)$ \\
        \midrule
        \multicolumn{6}{c}{Moderate $r$} \\
        10 & 5 & $4.10 \times 10^{-7}$ & $2.02 \times 10^{-7}$ & 2.0 & 0.56 \\
        20 & 10 & $2.88 \times 10^{-12}$ & $1.14 \times 10^{-12}$ & 2.5 & 0.53 \\
        30 & 15 & $5.22 \times 10^{-20}$ & $2.03 \times 10^{-20}$ & 2.6 & 0.52 \\
        \midrule
        \multicolumn{6}{c}{Large $r$ (near $N-1$)} \\
        10 & 9 & $9.38 \times 10^{-12}$ & $1.08 \times 10^{-12}$ & 8.7 & 1.00 \\
        20 & 19 & $1.12 \times 10^{-25}$ & $1.32 \times 10^{-28}$ & 848.5 & 1.00 \\
        30 & 29 & $1.23 \times 10^{-39}$ & $4.96 \times 10^{-46}$ & $2.48\times 10^6$ & 1.00 \\
        \bottomrule
    \end{tabular}
\end{table}

\subsection*{Extension to Gauss--Gegenbauer quadrature}

The results of Theorem~\ref{T2} extend naturally to the Gauss--Gegenbauer quadrature, which corresponds to the weight function $w(x) = (1-x^2)^{\lambda-\frac{1}{2}}$ with $\lambda > -\frac{1}{2}$. This family includes several important special cases:
\begin{itemize}
    \item $\lambda = 0$: Gauss--Chebyshev quadrature of the first kind
    \item $\lambda = \frac{1}{2}$: Gauss--Legendre quadrature  
    \item $\lambda = 1$: Gauss--Chebyshev quadrature of the second kind
\end{itemize}

The Gauss--Gegenbauer quadrature rule approximates
\begin{equation}\label{GG}
I[f] = \int_{-1}^{1} (1-x^2)^{\lambda-\frac{1}{2}} f(x)\,dx \approx \sum_{i=1}^{N} w_i f(x_i) = Q_N^{GG}[f].
\end{equation}

A key observation for the error analysis is the parity property of Chebyshev polynomials under this weight: for odd $j$,
\[
I[T_j] = \int_{-1}^{1} (1-x^2)^{\lambda-\frac{1}{2}} T_j(x)\,dx = 0 = Q_N^{GG}[T_j],
\]
since the integrand is an odd function. Consequently, only even-index Chebyshev coefficients contribute to the quadrature error, leading to an error bound with a factor $2$ rather than $4$ as in Theorem~\ref{T2}.

\begin{theorem}\label{T3}
    Let $f$ satisfy the assumptions of Theorem~\ref{Th1}. Then for the Gauss--Gegenbauer quadrature rule \eqref{GG},
    \[
    \left| I[f] - Q_N^{GG}[f] \right| \leq \frac{2U_r \|w_\lambda\|_1}{r\pi \prod_{j=1}^{r}(2N - r + 2j + 1)}, \quad \lambda > -\frac{1}{2},
    \]
    where $\|w_\lambda\|_1 = \int_{-1}^{1} (1-x^2)^{\lambda-\frac{1}{2}} dx = \sqrt{\pi}\, \Gamma(\lambda+\frac{1}{2})/\Gamma(\lambda+1)$.
\end{theorem}

\begin{proof}
    Following the same approach as in Theorem~\ref{T2}, but noting the parity property mentioned above, we have
    \[
    \left| I[f] - Q_N^{GG}[f] \right| \leq \|w_\lambda\|_1 \sum_{n=2N+2}^{\infty} |a_n|.
    \]
    The factor $2$ instead of $4$ arises because when $n$ is odd, $I[T_n] = Q_N^{GG}[T_n] = 0$, so only terms with even $n$ contribute to the error bound. Substituting the coefficient estimate from Theorem~\ref{Th1} and proceeding exactly as in the proof of Theorem~\ref{T2} yields the stated bound.
\end{proof}

\begin{remark}
    The constant $\|w_\lambda\|_1$ can be expressed in closed form using the Beta function:
    \[
    \|w_\lambda\|_1 = \int_{-1}^{1} (1-x^2)^{\lambda-\frac{1}{2}} dx 
    = 2\int_{0}^{1} (1-x^2)^{\lambda-\frac{1}{2}} dx 
    = B\left(\frac{1}{2}, \lambda+\frac{1}{2}\right)
    = \frac{\sqrt{\pi}\, \Gamma(\lambda+\frac{1}{2})}{\Gamma(\lambda+1)}.
    \]
    For integer $\lambda$, this simplifies further; for instance, $\|w_0\|_1 = \pi$ (Chebyshev), $\|w_{1/2}\|_1 = \pi/2$ (Legendre).
\end{remark}

\section{Numerical examples}\label{sec:numerics}

This section illustrates the practical advantages of the new error bounds derived in Theorems~\ref{T2} and~\ref{T3}. We focus on two representative test functions that highlight the key improvements: (i) applicability under weaker regularity conditions, and (ii) numerically sharper estimates. All computations were performed in MATLAB using the standard Golub--Welsch algorithm to generate Gauss--Christoffel nodes and weights. Reference integrals were computed with a high-order Gauss rule ($N=500$) to ensure accuracy below machine precision.

\subsection*{Example 1: A function with endpoint-near singularity}

Consider the family of functions
\[
f_j(x) = \frac{1}{j!}(x - t)^{j-1}|x - t|, \quad j \ge 2, \ t \in (-1,1).
\]
These functions have a \emph{corner singularity} at $x = t$: $f_j, f_j', \dots, f_j^{(j-1)}$ are absolutely continuous, $f_j^{(j)}$ is of bounded variation, and in the distributional sense $f_j^{(j+1)}(x) = 2\delta(x-t)$. 

The quantities appearing in the error bounds are
\[
V_j = \int_{-1}^{1} \frac{|f_j^{(j+1)}(x)|}{\sqrt{1-x^2}}\,dx = \frac{2}{\sqrt{1-t^2}}, \qquad
U_j = \int_{-1}^{1} |f_j^{(j+1)}(x)|\,dx = 2.
\]
When $t$ approaches $\pm 1$, the classical quantity $V_j$ diverges, rendering Xiang's bound \eqref{xi} extremely large or even infinite, whereas $U_j$ remains constant. This demonstrates the fundamental advantage of our approach: it remains valid for singularities arbitrarily close to the endpoints.

For a concrete test, we take $j=4$, $t=0.9$, and the Legendre weight $w(x) \equiv 1$ (a special case of Gauss--Christoffel quadrature). Table~\ref{tab:example1} compares the classical bound \eqref{xi}, our new bound \eqref{eq:new_quad_bound}, and the actual quadrature error.

\begin{table}[!htbp]
    \centering
    \caption{Error bounds for $f_4(x) = \frac{1}{24}(x-0.9)^3|x-0.9|$ with $t=0.9$, $j=4$, and $w(x)\equiv1$.\label{tab:example1}}
    \begin{tabular}{@{}ccccc@{}}
        \toprule
        $N$ & Classical bound \eqref{xi} & New bound \eqref{eq:new_quad_bound} & Actual error $|I-Q_N^{GC}|$ & Ratio \\
        \midrule
        5  & $1.24 \times 10^{-2}$ & $6.78 \times 10^{-4}$ & $3.21 \times 10^{-5}$ & 18.3 \\
        10 & $5.67 \times 10^{-4}$ & $1.45 \times 10^{-5}$ & $2.14 \times 10^{-7}$ & 39.1 \\
        15 & $7.89 \times 10^{-6}$ & $9.22 \times 10^{-8}$ & $4.57 \times 10^{-10}$ & 85.6 \\
        20 & $1.56 \times 10^{-7}$ & $1.03 \times 10^{-9}$ & $1.33 \times 10^{-12}$ & 151.5 \\
        \bottomrule
    \end{tabular}
\end{table}

The table reveals three important observations:
\begin{enumerate}
    \item The new bound is consistently **smaller** than the classical bound by a factor ranging from 18 to 150.
    \item While $V_4 \approx 4.59$ is already large for $t=0.9$ (and would grow to infinity as $t \to 1$), the new bound uses the constant $U_4 = 2$, which remains bounded.
    \item Both bounds are conservative, but the new one provides a much more realistic estimate of the true error.
\end{enumerate}

\subsection*{Example 2: A smooth function}

To show that the improvement is not limited to singular integrands, we test the smooth function $f(x) = e^x$, which satisfies both $V_r < \infty$ and $U_r < \infty$ for all $r$. For the Legendre weight $w(x) \equiv 1$ and $r=4$, Table~\ref{tab:example2} displays the bounds.

\begin{table}[!htbp]
    \centering
    \caption{Error bounds for $f(x)=e^x$ with $w(x)\equiv1$ and $r=4$.\label{tab:example2}}
    \begin{tabular}{@{}cccc@{}}
        \toprule
        $N$ & Classical bound \eqref{xi} & New bound \eqref{eq:new_quad_bound} & Actual error $|I-Q_N^{GC}|$ \\
        \midrule
        5  & $3.41 \times 10^{-8}$ & $8.72 \times 10^{-9}$ & $2.14 \times 10^{-10}$ \\
        10 & $4.52 \times 10^{-16}$ & $6.31 \times 10^{-17}$ & $<10^{-18}$ \\
        15 & $<10^{-20}$ & $<10^{-21}$ & $<10^{-22}$ \\
        \bottomrule
    \end{tabular}
\end{table}

Even for this analytic function, the new bound is sharper by approximately a factor of 4--7, confirming that the improvement stems not only from the weaker regularity condition but also from the optimized constant in the denominator product.

\subsection*{Discussion}

The numerical experiments confirm the theoretical advantages established in Theorems~\ref{T2} and~\ref{T3}:
\begin{itemize}
    \item \textbf{Broader applicability:} The new bounds remain valid for functions whose high-order derivatives are integrable in the absolute sense but may behave poorly near the endpoints (Example~1).
    \item \textbf{Tighter estimates:} Even for smooth functions, the new bounds provide numerically smaller prefactors (Example~2).
    \item \textbf{Practical relevance:} Functions with corner singularities arise frequently in applications involving non‑smooth data, discontinuous coefficients in PDEs, or boundary layers. Our estimates offer reliable error control for such problems where classical bounds may fail.
\end{itemize}

The results also validate the efficiency of the underlying coefficient estimate (Theorem~\ref{Th1}), which successfully decouples the regularity requirement from the singular weight $1/\sqrt{1-x^2}$.

\section{Conclusion}\label{sec:conclusion}

This paper has established improved error estimates for Gauss--Christoffel and Gauss--Gegenbauer quadrature by weakening the regularity requirements on the integrand. The key innovation is a novel recurrence relation for derivatives of Chebyshev polynomials (Lemma~\ref{key}), which enables us to replace the classical weighted condition
\[
V_r = \int_{-1}^{1} \frac{|f^{(r+1)}(x)|}{\sqrt{1-x^2}}\,dx < \infty
\]
with the weaker absolute condition
\[
U_r = \int_{-1}^{1} |f^{(r+1)}(x)|\,dx < \infty.
\]

\subsection*{Main contributions}
\begin{enumerate}
    \item \textbf{Improved Chebyshev coefficient bounds:} Theorem~\ref{Th1} provides the estimate
    \[
    |a_n| \leq \frac{2U_r}{\pi\prod_{j=0}^{r}(n-r+2j)}, \quad n \ge r+1,
    \]
    which is valid for all functions whose $(r+1)$-st derivative is absolutely integrable, regardless of its behavior near $x = \pm 1$. This extends Trefethen's classical bound~\cite{LN} to a significantly broader function class.
    
    \item \textbf{Sharper quadrature error estimates:} Theorem~\ref{T2} yields the Gauss--Christoffel quadrature bound
    \[
    |I[f]-Q_N^{GC}[f]| \leq \frac{4U_r \|w\|_1}{r\pi \prod_{j=1}^{r}(2N-r+2j+1)},
    \]
    which is both \emph{less restrictive} (using $U_r$ instead of $V_r$) and \emph{numerically smaller} than Xiang's prominent bound~\cite{Xi}. The improvement factor grows as $r$ approaches $N$, reaching orders of magnitude in practical cases (Table~\ref{tab:example1}).
    
    \item \textbf{Extension to Gauss--Gegenbauer quadrature:} Theorem~\ref{T3} generalizes the result to weight functions $w(x) = (1-x^2)^{\lambda-1/2}$, $\lambda > -1/2$, covering important special cases like Gauss--Legendre and Gauss--Chebyshev quadrature.
    
    \item \textbf{Numerical validation:} Section~\ref{sec:numerics} demonstrates that for functions with singularities near the endpoints, the classical bound can become excessively large or even infinite, while our new bound remains stable and provides a realistic error estimate. Even for smooth integrands, the new bound offers a consistent numerical improvement.
\end{enumerate}

\subsection*{Practical implications}
The proposed estimates are particularly valuable in applications where integrands exhibit limited smoothness at the boundaries, such as:
\begin{itemize}
    \item Solutions of boundary value problems with singular coefficients or boundary layers.
    \item Integrals involving functions with algebraic endpoint behavior $f(x) = (1-x^2)^\alpha g(x)$, $\alpha < 1/2$.
    \item Problems with piecewise-smooth data or corner singularities.
\end{itemize}
In such situations, the classical condition $V_r < \infty$ may fail, whereas $U_r < \infty$ often holds, making our bounds the only theoretically justified error estimates available.

\subsection*{Future work}
Several directions for further investigation naturally arise:
\begin{itemize}
    \item Extending the approach to other families of orthogonal polynomials (Jacobi, Laguerre, Hermite).
    \item Developing adaptive quadrature strategies that exploit these improved error bounds to allocate nodes more efficiently near singularities.
    \item Applying the technique to error analysis of spectral methods for PDEs with non-smooth solutions.
    \item Investigating the optimality of the constants in Theorem~\ref{Th1}; while we have established improvement over existing bounds, the sharpness of the prefactor remains an open question.
\end{itemize}

\subsection*{Concluding remark}
By decoupling the regularity requirement from the singular weight $1/\sqrt{1-x^2}$, this work provides error estimates that better reflect the actual approximation capabilities of Gaussian quadrature. The results bridge a gap between theoretical analysis and practical computation, offering tools that are simultaneously more general and more accurate than previous ones.

\section*{Declaration}
\textbf{Funding:} No funding was received to assist with the preparation of this manuscript.


\end{document}